\documentclass[11pt, a4paper, reqno]{amsart}

\usepackage{amsfonts,amsmath,amsthm}
\usepackage{amssymb}
\usepackage{latexsym}
\usepackage{mathrsfs}
\usepackage{amsthm}
\usepackage{enumitem}
\usepackage{mathtools}
\usepackage{moreenum}
\usepackage{amssymb,verbatim}
\usepackage{amsmath}
\usepackage{amssymb}
\usepackage{amsfonts}
\usepackage{etoolbox}

\usepackage[english]{babel}

\usepackage[dvipsnames]{xcolor}

\usepackage[colorlinks,citecolor=blue,urlcolor=black,linkcolor=blue!90!red]{hyperref}

\theoremstyle{plain}
\newtheorem{theorem}{Theorem}[section]
\newtheorem{lemma}[theorem]{Lemma}

\theoremstyle{definition}
\newtheorem{definition}[theorem]{Definition}
\newtheorem{example}[theorem]{Example}

\theoremstyle{remark}
\newtheorem*{remark}{Remark}

\newcommand{\N}{{\mathbb N}}

\newcounter{abc}   
\newcounter{iiiii} 

\newenvironment{aequivalenz}
{\setcounter{iiiii}{0}
\begin{list}%
{{\rm (\roman{iiiii})}}
{\usecounter{iiiii}
\parsep=0pt plus 1pt
\topsep=1pt plus 2pt minus 1pt
\itemsep=1pt plus 2pt minus 1pt
\leftmargin=3\baselineskip \labelsep=.6\baselineskip
\labelwidth=2.4\baselineskip
\rightmargin 0pt}%
}
{\end{list}}

\newenvironment{statements}%
{\setcounter{abc}{0}
\begin{list}%
{{\rm (\alph{abc})}}
{\usecounter{abc}
\parsep=0pt plus 1pt
\topsep=1pt plus 2pt minus 1pt
\itemsep=1pt plus 2pt minus 1pt
\leftmargin=3\baselineskip \labelsep=.6\baselineskip
\labelwidth=2.4\baselineskip
\rightmargin 0pt}%
}
{\end{list}}

{\begin{list}%
{$\bullet$}
{\leftmargin=3\baselineskip
\labelsep=\baselineskip \labelwidth=2.5\baselineskip
\parsep=0pt plus 1pt
\topsep=1pt plus 2pt minus 1pt
\itemsep=1pt plus 2pt minus 1pt
\rightmargin 0pt}%
}
{\end{list}}

\newcommand{\bea}{\begin{eqnarray*}}
\newcommand{\eea}{\end{eqnarray*}}
\newcommand{\beq}{\begin{equation}}
\newcommand{\eeq}{\end{equation}}
\newcommand{\begsta}{\begin{statements}}
\def\endsta{\end{statements}}
\newcommand{\begaeq}{\begin{aequivalenz}}
\def\endaeq{\end{aequivalenz}}

\numberwithin{equation}{section}

\begin{document}

\title[Fixed points of uniformly asymptotically H\"older-Lipschitz maps]
{Fixed points for uniformly asymptotically H\"older-Lipschitz mappings}

\author{C. S. R. da Silva and L. G. Josino}

\date{\today}



 \thanks{The authors were partially
 supported by CAPES-Brazil.}

\address{Departamento de Matemática,
Universidade Federal do Ceará,
 Campus do Pici, 60455-360, Fortaleza-CE, Brasil}
\email{csergiomat@gmail.com; lu-gauss@hotmail.com}

\begin{abstract}
In this work, we show that if $K$ is a nonempty, closed, bounded, and
convex subset of $X$ and $1\leq \gamma<\sqrt{N(X)},$ where $N(X)$ denotes the
Bynum normal structure coefficient, then every uniformly asymptotically H\"older
$\gamma$-Lipschitz mapping $T\colon K\to K$ has a fixed point.
To achieve this, it was necessary to guarantee the existence of a subsequence of a bounded sequence whose asymptotic diameter is controlled by the quantity $D(\{x_n\}),$ introduced by Jiménez-Melado,
while simultaneously preserving the asymptotic radii at every point of the closure of the linear span generated by the original sequence. This allowed us to remove the H\"older exponent from the calculations and obtain a H\"older-type extension of the Casini--Maluta, Lim--Xu, and, consequently, Kim--Xu theorems.
\end{abstract}

\maketitle

\thispagestyle{empty}

\section{Introduction}

Let $X$ be a Banach space and let $K$ be a nonempty, closed, bounded, and
convex subset of $X$. The fixed point property for mappings whose iterates satisfy
uniform metric estimates is closely related to the geometry of the
space under consideration. Goebel and Kirk proved in 1972 that
uniformly convex spaces have the fixed point property for asymptotically nonexpansive mappings (see \cite{GoebelKirk1972}). Recall that a mapping $T:K\to K$ is asymptotically nonexpansive if, for every natural number $n\ge1$ and for every $x,y\in K$, $\|T^n(x)-T^n(y)\|\le k_n\|x-y\|$, where $(k_n)_{n\in\mathbb{N}}$ is a sequence such that $k_n\to1$. In 1973, the same authors considered the class of
uniformly $\gamma$-Lipschitzian mappings $T\colon K\to K$, characterized by
\[
\|T^n(x)-T^n(y)\|\leq \gamma\|x-y\|,
\qquad x,y\in K,\quad n\geq 1.
\]
They proved that such a mapping has a fixed point provided that

\[
\gamma\left(1-\delta_X\left(\frac{1}{\gamma}\right)\right)<1,
\]
where $\delta_X$ is the modulus of
convexity of $X$ (see \cite{GoebelKirk1973}). In another direction, in 1984, Maluta, while investigating the uniform normal
structure of Banach spaces, proved, in particular, that every Banach space
with uniform normal structure is reflexive \cite[Corollary 4.1]{maluta1984uniformly}. In the following year, Casini and Maluta \cite{CM}
proved that a uniformly
$\gamma$-Lipschitzian mapping has a fixed point whenever $X$ has uniform normal
structure and
\[
1\leq \gamma<\sqrt{N(X)},
\]
where $N(X)$ denotes the normal structure coefficient introduced by Bynum
\cite{Bynum}.
The normal structure coefficient of $X$ is defined by
\[
N(X)
=
\inf\left\{
\frac{\mathsf{d}(E)}{r(E)}
:
E\in\mathcal{B}(X) \text{ and } \mathsf{d}(E)>0
\right\},
\]
where $\mathcal{B}(X)$ denotes the family of all nonempty,
closed, bounded, and convex subsets of $X$, $\mathsf{d}(E)$ is the diameter of $E$, and
$r(E)$ is the Chebyshev radius of $E$ relative to itself. We also define
\[
    \widetilde{N}(X)=[N(X)]^{-1}.
\]
Consequently, $X$ has uniform normal structure if, and only if,
$\widetilde{N}(X)<1$, or, equivalently, $N(X)>1$.

The Casini--Maluta Theorem was subsequently extended by Lim and Xu
\cite{LimXu}, in which the closedness and convexity conditions on the domain of the mapping were removed, followed by Kim and Xu \cite{KimXu}, who combined this extension with a minimal-set argument to show that every Banach space with uniform normal structure has the fixed point property for asymptotically nonexpansive mappings. Thus, the results of Casini--Maluta,
Lim--Xu, and Kim--Xu form a natural progression of the fixed point property from uniformly Lipschitz mappings to asymptotically nonexpansive mappings.

This line of investigation was subsequently considered under pointwise asymptotic conditions. Domínguez Benavides and Lorenzo Ramírez \cite{Benavides2022Fixed} observed, in particular, that previous arguments for mappings of asymptotically nonexpansive type in spaces with uniform normal structure actually yield the existence of an orbit ($\{T^n x\}$) converging to the point itself ($x$), which makes it possible to obtain fixed point results for pointwise classes without requiring the global continuity of an iterate. The authors advanced in this direction by establishing fixed point results without the hypothesis of the existence of a continuous iterate. In particular, they proved that every pointwise asymptotically nonexpansive mapping defined on a nonempty, closed, bounded, and convex subset of a nearly uniformly convex (NUC) Banach space has a fixed point. In addition, they showed that pointwise eventually nonexpansive mappings have a fixed point when the asymptotic centers of sequences in the domain are compact.

On the other hand, in the context of H\"older estimates, Barroso and Ferreira \cite{BarrosoFerreira} subsequently used
retraction methods to construct fixed-point-free mappings with
null minimal displacement on unit balls and studied, in particular,
uniform H\"older estimates for the iterates with a fixed exponent
$\alpha\in(0,1)$. We recommend the works \cite{Benavides2022Fixed} --
\cite{espinola2024orbital} for a deeper treatment of this line of research.

More recently, Barroso and da Silva
\cite{BarrosoSilva} moved from a fixed exponent to sequences
$(\alpha_n)\subset(0,1]$ satisfying $\alpha_n\to1$. Their work introduced and
studied mappings of asymptotically H\"older nonexpansive type,
established fixed point results under the condition $\varepsilon_0(X)<1$, and
presented examples showing that the resulting classes genuinely extend
their classical counterparts.

\subsection{Objective}

The objective of this article is to develop, in the H\"older setting, a counterpart of the
fixed point results of Casini--Maluta, Lim--Xu, and Kim--Xu. More
precisely, we study the following classes of mappings:

\begin{definition}\label{def:holder-classes}\cite{BarrosoSilva}
Let $K$ be a nonempty subset of a Banach space $X$, and let
$T\colon K\to K$.
\begin{enumerate}
    \item[(i)] The mapping $T$ is called \emph{uniformly asymptotically
    H\"older $\gamma$-Lipschitzian}, where $\gamma\geq 1$, if there exists a
    sequence $(\alpha_n)\subset(0,1]$, with $\alpha_n\to1$, such that
    \begin{equation}
        \|T^n x-T^n y\|
        \leq
        \gamma\|x-y\|^{\alpha_n}
        \label{eq:uniform-holder}
    \end{equation}
    for all $x,y\in K$ and every $n\geq1$.

    \item[(ii)] The mapping $T$ is called \emph{asymptotically H\"older
    nonexpansive} if there exist sequences $(k_n)\subset[0,\infty)$ and
    $(\alpha_n)\subset(0,1]$, with $k_n\to0$ and $\alpha_n\to1$, such that
    \begin{equation}
        \|T^n x-T^n y\|
        \leq
        (1+k_n)\|x-y\|^{\alpha_n}
        \label{eq:asymptotically-holder}
    \end{equation}
    for all $x,y\in K$ and every $n\geq1$.
\end{enumerate}
\end{definition}

It is important to emphasize that the class of uniformly asymptotically H\"older $\gamma$-Lipschitz mappings strictly contains the class of uniformly $\gamma$-Lipschitz mappings. Indeed, if $\alpha_n=1$ for every $n$, definitions (i) and (ii) above reduce, respectively, to the classical notions of uniformly $\gamma$-Lipschitz mapping and asymptotically nonexpansive mapping. On the other hand, one can construct mappings that are not even Lipschitzian but satisfy condition (i) in \ref{def:holder-classes}, as we shall see in Example \ref{ex:non-lips}.

\subsection{Main results}
The first result of this work is an extension of the Casini--Maluta Theorem \cite[Theorem 3.1]{CM} to the class of uniformly asymptotically H\"older $\gamma$-Lipschitz mappings.

\begin{theorem}[Main Result]\label{thm:main1}
Let $X$ be a Banach space with uniform normal structure, and let $K$ be a
nonempty, closed, bounded, and convex subset of $X$. If $T\colon K\to K$
is a uniformly asymptotically H\"older
$\gamma$-Lipschitzian mapping and
\[
1\leq\gamma<\sqrt{N(X)},
\]
then $T$ has a fixed point.
\end{theorem}

The proof of Theorem \ref{thm:main1} cannot be obtained from a mere direct adaptation of the proof given by Casini--Maluta in \cite[Theorem 3.1]{CM}. Indeed, if $n>m,$ then
$$
\|T^nx-T^mx\|\leq\gamma\|T^{n-m}x-x\|^{\alpha_m}.
$$
The problem arises when trying to control directly the asymptotic diameter of the sequence of iterates $\{T^nx\}$
$$
\mathsf d_a(\{T^nx\})=\lim_{k\to\infty}\sup_{n,m\geq k}\|T^nx-T^mx\|.
$$
Even though $m\to\infty$ implies that $\alpha_m\to 1,$ the condition $n,m\geq k$ does not guarantee that $n-m\to \infty.$ For example, we may take $n=m+1$ for every $m$. Thus, the term $\|T^{n-m}x-x\|$ may involve small indices and, consequently, cannot be controlled by the asymptotic radius
$$
r_a(\{T^jx\},x)=\limsup_{j\to\infty}\|T^jx-x\|.
$$

To overcome this difficulty, it was necessary to consider the quantity
$$
    D(\{T^nx\})
    =
    \limsup_{m\to\infty} \limsup_{n\to\infty} \| T^nx - T^mx \|.
$$
Observe that the order of the limits is essential in our case. For each fixed $m$, we have $n-m\to\infty$ as $n\to\infty.$
With this definition, we have

\begin{align*}
    \limsup_{n\to\infty}\|T^nx-T^mx\|&\leq\gamma\limsup_{n\to\infty}\|T^{n-m}x-x\|^{\alpha_m}\\ 
    &=\gamma\left(\limsup_{n\to\infty}\|T^{n-m}x-x\|\right)^{\alpha_m}.
\end{align*}
Taking the upper limit as $m\to\infty$ and using $\alpha_m\to1$, it follows that
$$
D(\{T^n x\}) \leq \gamma\limsup_{m\to\infty}r_a(\{T^n x\},x)^{\alpha_m} = \gamma r_a(\{T^n x\},x).
$$

However, the Casini--Maluta lemma \cite[Lemma 3.1]{CM} cannot be used here, since it provides a control for the asymptotic radius in terms of $\mathsf d_a$, and not in terms of $D$. Thus, it was necessary to develop a new technique allowing us to obtain this estimate. We show that every
bounded sequence $\{x_n\}$ admits a subsequence $\{x_{n_j}\}$ whose
asymptotic diameter is bounded by
$$
D(\{x_n\})=\limsup_{m\to\infty}\limsup_{n\to\infty}\|x_n-x_m\|,
$$
while preserving, at the same time,
the asymptotic radius of $\{x_n\}$ at every point of $Y=\overline{\operatorname{span}}\{x_n:n\in\mathbb N\}.$

Combining this new technical subsequence lemma with the original Casini--Maluta lemma, we obtain a point $z\in\overline{\operatorname{conv}}\{x_n:n\in\mathbb N\}$ such that
$$
    r_a(\{x_n\},z)
    \leq
    \widetilde N(X)D(\{x_n\})
    \quad \text{and}\quad 
    \|z-y\|
    \leq
    r_a(\{x_n\},y),
    \quad \forall y\in X.
$$
With this version of the Casini--Maluta lemma, now in terms of $D$, it was possible to prove Theorem \ref{thm:main1}.

Next, we obtain the corresponding H\"older extension of the Lim--Xu theorem \cite[Theorem 1]{LimXu}.

\begin{theorem}\label{thm:main02}
Let $X$ be a Banach space with uniform normal structure, let $K$ be a
nonempty and bounded subset of $X$, and let $T\colon K\to K$ be a
uniformly asymptotically H\"older $\gamma$-Lipschitzian mapping with
\[
1\leq\gamma<\sqrt{N(X)}.
\]
Suppose that there exists a nonempty, closed, bounded, and convex subset
$E\subset K$ such that
\begin{equation*}\tag{P}\label{prop:P}
    x\in E
    \quad\Longrightarrow\quad
    \omega_w(x)\subset E,
\end{equation*}
where the weak $\omega$-limit set of $T$ at $x$ is defined by
\[
    \omega_w(x)
    =
    \left\{
        y\in X:
        y=\operatorname{weak\text{-}lim}_{j\to\infty}T^{n_j}x
        \text{ for some sequence } n_j\to\infty
    \right\}.
\]
Then $T$ has a fixed point in $E$.
\end{theorem}

Combining Theorem \ref{thm:main02} with a minimal-set argument,
we obtain a H\"older version of the Kim--Xu theorem \cite[Theorem 1]{KimXu}.

\begin{theorem}\label{thm:main03}
Let $X$ be a Banach space with uniform normal structure, let $K$ be a
nonempty, closed, bounded, and convex subset of $X$, and let
$T\colon K\to K$ be an asymptotically H\"older nonexpansive mapping. Then
$T$ has a fixed point.
\end{theorem}

Thus, the contribution of this work comprises two complementary parts. The
first is the subsequence principle and the resulting version, in terms of $D$,
of the Casini--Maluta lemma; the second consists of the application of this technique
to obtain H\"older-type extensions of the Casini--Maluta, Lim--Xu,
and Kim--Xu theorems.

\subsection{Additional comments}

Theorems~\ref{thm:main1}, \ref{thm:main02}, and \ref{thm:main03} recover the
results of Casini--Maluta, Lim--Xu, and Kim--Xu, respectively, when
$\alpha_n=1$ for every $n$. The examples presented below show that the
classes considered here contain genuinely non-Lipschitzian mappings and
that the condition $\alpha_n\to1$ cannot, in general, be replaced by a
fixed exponent $\alpha<1$. In addition, since $\varepsilon_0(X)<1$ implies
uniform normal structure, Theorem~\ref{thm:main03} extends the fixed point
conclusion of Barroso and da Silva \cite{BarrosoSilva} to a broader class
of Banach spaces.

The paper is organized as follows. Section
\ref{sec:preliminares} contains the remaining definitions and preliminary lemmas
used throughout the paper. Section \ref{sec:provas} is devoted to the proofs of our main results. Section \ref{sec:exemplos} presents examples illustrating the usefulness
of Lemma \ref{lem:extraction}, the scope of the classes of mappings considered,
and the role of the condition $\alpha_n\to 1$.


\section{Preliminaries}\label{sec:preliminares}

To properly contextualize the purposes of this work, we begin this section by recalling some fundamental geometric notions and notation relevant to the underlying topic. Given subsets $A$ and $B$ of $X$, we recall that: 
\begin{itemize}\setlength\itemsep{2mm}
\item $\mathsf{d}(A) = \sup\{ \|x - y\| \colon x,y \in A\}$ is the diameter of $A$.

\item For each $x \in X$, $r(A,x) = \sup\{ \|x - y\| \colon y \in A\}$ is the Chebyshev radius of $A$ centered at $x$.

\item $r(A,B) = \inf\{ r(A,x) \colon x \in B\}$ is the Chebyshev radius of $A$ relative to $B$.

\item $r(A) = r(A,A)$ is the Chebyshev radius of $A$.
\end{itemize}
Furthermore, for a bounded sequence $\{x_n\}\subset X$, we recall the following notions:

\begin{itemize}\setlength\itemsep{2mm}
    \item $\mathsf{d}_a(\{x_n\})=\displaystyle\lim_{k\to\infty}\sup\{\|x_n-x_m\|: n,m\geq k\}$ is the asymptotic diameter;

    \item $r_a(\{x_n\},x)=\displaystyle\limsup_{n\to\infty}\|x_n-x\|$ is the asymptotic radius of the sequence relative to the point $x\in X$;

    \item $r_a(\{x_n\},A)=\inf\{r_a(\{x_n\},x):\,x\in A\}$ is the asymptotic radius of the sequence relative to the set $A$;
\end{itemize}

\noindent Finally, the following quantity will play a fundamental role in our approach. It was introduced by Jimenez-Melado in \cite{Jimenez}: 

\begin{definition} For a bounded sequence  $\{x_n\}$ in a Banach space $X$, define:

\[
D(\{x_n\})=\limsup_{m\to\infty}\limsup_{n\to\infty}\|x_n-x_m\|.
\]
\end{definition}
\begin{remark}It is easy to see that $D(\{x_n\})$ is well defined and finite. Indeed, if $\|x_n\|\leq M$ for every $n\in\mathbb N$, then $0 \leq D(\{x_n\}) \leq 2M.$
\end{remark}

The relevance of this object lies in the order of the limits in its definition. Notice that the upper limit as $n\to\infty$ gives the asymptotic radius of the sequence relative to one of its own points $x_m$, followed by the upper limit as $m\to\infty$, which is precisely the index that orders the sequence of H\"older exponents $(\alpha_m)$.

The following Lemma \ref{lem:extraction} is the main ingredient that allows us to adjust the Casini--Maluta Lemma \cite[Lemma 3.1]{CM}.
\begin{lemma}\label{lem:extraction}
If $\{x_n\}_{n\in\N}$ is a bounded sequence in $X$ and $$Y=\overline{\mathrm{span}}\{x_n;\;n\in\N\},$$ then there exists a subsequence $\{x_{n_j}\}$ such that:
\begin{itemize}
    \item [i)]$\mathsf{d}_a(\{x_{n_j}\})\leq D(\{x_n\})$
    \item [ii)]$r_a(\{x_{n_j}\},y)=r_a(\{x_n\},y),\, \quad\forall y\in Y.$
\end{itemize}
\end{lemma}
\begin{proof}
Define \(R\colon X\to[0,\infty)\) by
\[
    R(y)
    =
    \limsup_{n\to\infty}\|x_n-y\|
    =
    r_a((x_n),y).
\]
First, note that since $\{x_n\}$ is a bounded sequence, we have $R(y) < \infty$ for all $y \in X$. For simplicity, write
$$
    D:=D(\{x_n\})=\limsup_{m\to\infty}R(x_m).
$$
Since $Y=\overline{\operatorname{span}}\{x_n:n\in\mathbb N\}$ is separable, there exists a dense sequence
$\{y_q\}_{q\in\mathbb N}$ in $Y$. Choose a strictly decreasing
sequence $(\varepsilon_j)$ of positive numbers such that
$\varepsilon_j\to0$, and a sequence $(q_j)$ in $\mathbb N$ such
that every positive integer occurs infinitely many times.
By the definition of $D$, for every $j\in\mathbb N$, there exists
$M_j\in\mathbb N$ such that
$$
    m \geq M_j \Rightarrow R(x_m)<D+\frac{\varepsilon_j}{2}.
$$
We first construct \(n_1\) and \(p_1\). Since
$R(y_{q_1})=\limsup_{n\to\infty}\|x_n-y_{q_1}\|,$
there exist arbitrarily large \(n\in\mathbb N\) such that
\[
    \|x_n-y_{q_1}\|
    >
    R(y_{q_1})-\varepsilon_1.
\]
Thus, we may choose \(n_1\geq M_1\) such that
\[
    \|x_{n_1}-y_{q_1}\|
    >
    R(y_{q_1})-\varepsilon_1.
\]
Since \(n_1\geq M_1\), we also have
\[
    R(x_{n_1})
    <
    D+\frac{\varepsilon_1}{2}.
\]
Therefore, there exists \(p_1>n_1\) such that
\[
    n\geq p_1
    \quad\Rightarrow\quad
    \|x_n-x_{n_1}\|<D+\varepsilon_1.
\]
Suppose that
$n_1<\cdots<n_{j-1}$ and $p_1,\ldots,p_{j-1}$ have already been chosen
so that, for every $i<j$,
\[
n\geq p_i
\quad\Rightarrow\quad
\|x_n-x_{n_i}\|<D+\varepsilon_i.
\]
Since
$R(y_{q_j})=\limsup_{n\to\infty}\|x_n-y_{q_j}\|$, there are arbitrarily
large $n$ such that
\[
\|x_n-y_{q_j}\|>R(y_{q_j})-\varepsilon_j.
\]
Thus, we may choose
\[
n_j>\max\{n_{j-1},M_j,p_1,\ldots,p_{j-1}\}
\]
such that
\[
\|x_{n_j}-y_{q_j}\|>R(y_{q_j})-\varepsilon_j.
\]
Since $n_j\geq M_j$,
$R(x_{n_j})<D+\varepsilon_j/2$. Hence, there exists \(p_j>n_j\) such that
\[
    n\geq p_j
    \quad\Rightarrow\quad
    \|x_n-x_{n_j}\|<D+\varepsilon_j.
\]
Set $u_j=x_{n_j}$. We first show that
$\mathsf{d}_a(\{u_j\})\leq D(\{x_n\})$. If $r>s$, then $n_r\geq p_s$, and hence
\[
\|u_r-u_s\|
=
\|x_{n_r}-x_{n_s}\|
<
D+\varepsilon_s.
\]
Therefore, if $r,s\geq k$, then
$\|u_r-u_s\|\leq D+\varepsilon_k$, and consequently
\[
\sup_{r,s\geq k}\|u_r-u_s\|
\leq D+\varepsilon_k.
\]
Letting $k\to\infty$, we obtain
\[
\mathsf{d}_a(\{u_j\})
=
\lim_{k\to\infty}\sup_{r,s\geq k}\|u_r-u_s\|
\leq D
=
D(\{x_n\}),
\]
which proves (i).\\
We now prove (ii). Fix $q\in\mathbb N$. Since $\{u_j\}$ is a
subsequence of $\{x_n\}$,
\[
r_a(\{u_j\},y_q)
\leq
r_a(\{x_n\},y_q)
=
R(y_q).
\]
On the other hand, since $q$ occurs infinitely many times in
$\{q_j\}$, there exist $j_1<j_2<\cdots$ such that $q_{j_\ell}=q$ for
every $\ell$. By the construction of $n_j$,
\[
\|u_{j_\ell}-y_q\|
=
\|x_{n_{j_\ell}}-y_{q_{j_\ell}}\|
>
R(y_q)-\varepsilon_{j_\ell}.
\]
Since $\varepsilon_{j_\ell}\to0$,
\[
\limsup_{\ell\to\infty}\|u_{j_\ell}-y_q\|
\geq R(y_q).
\]
As $\{u_{j_\ell}\}$ is a subsequence of $\{u_j\}$,
\[
r_a(\{u_j\},y_q)
\geq
\limsup_{\ell\to\infty}\|u_{j_\ell}-y_q\|
\geq R(y_q).
\]
Thus,
\[
r_a(\{u_j\},y_q)=r_a(\{x_n\},y_q)
\]
for every $q\in\mathbb N$.\\
Finally, define
\[
F(y)=r_a(\{x_n\},y)
\quad\text{and}\quad
G(y)=r_a(\{u_j\},y).
\]
Both $F$ and $G$ are $1$-Lipschitz. Indeed, by the triangle inequality,
\[
F(y)\leq F(z)+\|y-z\|
\quad\text{and}\quad
F(z)\leq F(y)+\|y-z\|,
\]
so that $|F(y)-F(z)|\leq\|y-z\|$. The same argument applies to $G$.\\
Let $y\in Y$. Since $\{y_q\}$ is dense in $Y$, there exists a
subsequence $\{y_{q_\ell}\}$ such that $y_{q_\ell}\to y$. Since
$F(y_{q_\ell})=G(y_{q_\ell})$ for every $\ell$,
\[
\begin{aligned}
|F(y)-G(y)|
&\leq
|F(y)-F(y_{q_\ell})|
+
|F(y_{q_\ell})-G(y_{q_\ell})|
+
|G(y_{q_\ell})-G(y)|\\
&\leq
2\|y-y_{q_\ell}\|.
\end{aligned}
\]
Letting $\ell\to\infty$, we obtain $F(y)=G(y)$. Hence
\[
r_a(\{u_j\},y)=r_a(\{x_n\},y)
\qquad\text{for every }y\in Y.
\]
Since $u_j=x_{n_j}$, the proof is complete.
\end{proof}

\begin{lemma}\label{lem:casini}
Let $X$ be a Banach space with uniform normal structure, and let
$\{x_n\}$ be a bounded sequence in $X$. Then there exists
$z\in\overline{\operatorname{conv}}\{x_n:n\in\mathbb N\}$ satisfying:

\begin{itemize}
    \item[(i)] $r_a(\{x_n\},z)\leq \widetilde N(X)D(\{x_n\});$
    \item[(ii)] $\|z-y\|\leq r_a(\{x_n\},y)$ for every $y\in X$.
\end{itemize}
\end{lemma}
\begin{proof}
Set $Y=\overline{\operatorname{span}}\{x_n:n\in\mathbb N\}$. By Lemma \ref{lem:extraction}, there exists a subsequence
$\{x_{n_j}\}=\{u_j\}$ of $\{x_n\}$ such that
$\mathsf{d}_a(\{u_j\})\leq D(\{x_n\})$ and
$r_a(\{x_n\},y)=r_a(\{u_j\},y)$ for every $y\in Y$.\\
Applying Lemma 3.1 of Casini--Maluta (see \cite{CM}) to the subsequence $\{u_j\}$,
we obtain
$z\in\overline{\operatorname{conv}}\{u_j:j\in\mathbb N\}$ such that
$$
r_a(\{u_j\},z)\leq \widetilde N(X)\mathsf{d}_a(\{u_j\})
$$
and
$$
\|z-y\|\leq r_a(\{u_j\},y)
\quad\text{for every }y\in X.
$$\\
Since $u_j=x_{n_j}\in Y$ for every $j\in\mathbb N$ and $Y$ is a
closed linear subspace of $X$, we have
$\overline{\operatorname{conv}}\{u_j:j\in\mathbb N\}\subset Y$.
Consequently, $z\in Y$. Hence, by Lemma \ref{lem:extraction},
$r_a(\{u_j\},z)=r_a(\{x_n\},z)$. Therefore,
\[
\begin{aligned}
r_a(\{x_n\},z)
&=r_a(\{u_j\},z)\\
&\leq \widetilde N(X)\mathsf{d}_a(\{u_j\})\\
&\leq \widetilde N(X)D(\{x_n\}),
\end{aligned}
\]
which proves (i).\\
To prove (ii), observe that, since $\{u_j\}$ is a subsequence of
$\{x_n\}$,
\[
\begin{aligned}
r_a(\{u_j\},y)
&=\limsup_{j\to\infty}\|u_j-y\|\\
&\leq \limsup_{n\to\infty}\|x_n-y\|\\
&=r_a(\{x_n\},y).
\end{aligned}
\]
Thus,
\[
\|z-y\|\leq r_a(\{u_j\},y)\leq r_a(\{x_n\},y)
\]
for every $y\in X$.\\
Finally, since $\{u_j\}$ is a subsequence of $\{x_n\}$,
$\overline{\operatorname{conv}}\{u_j:j\in\mathbb N\}
\subset\overline{\operatorname{conv}}\{x_n:n\in\mathbb N\}$.
Therefore,
$z\in\overline{\operatorname{conv}}\{x_n:n\in\mathbb N\}$.
\end{proof}



\section{Proof of main results}\label{sec:provas}
\begin{proof}[\textbf{Proof of Theorem \ref{thm:main1}}]
Consider the sequence $\{T^n(x)\}$, and let $z=z(x)$ be the point provided by Lemma \ref{lem:casini} for the sequence $\{T^n(x)\}$. Define
\[
d(x):=\limsup_{n\to\infty}\|T^n(x)-x\|.
\]
Condition \emph{(i)} of Lemma \ref{lem:casini} gives us
\[
r_a(\{T^n(x)\},z(x))\leq\tilde N(X)\cdot D(\{T^n(x)\}).
\]
Fix $m\in \mathbb{N}$. Then,
\[
    \begin{split}
    \limsup_{n\to\infty}\|T^nx-T^mx\|&\le\gamma\limsup_{n\to\infty}\|T^{n-m}x-x\|^{\alpha_m}\\
    &=\gamma\limsup_{j\to\infty}\|T^jx-x\|^{\alpha_m}\\ 
    &=\gamma d(x)^{\alpha_m}.
    \end{split}
\]
Thus,
    \[
    D(\{T^nx\})=\limsup_{m\to\infty}\limsup_{n\to\infty}\|T^nx-T^mx\|\le\limsup_{m\to\infty}\gamma d(x)^{\alpha_m}=\gamma d(x).
    \]
Therefore, $r_a(\{T^n(x)\},z(x))\leq\gamma\tilde N(X) d(x)$. Moreover, by condition \emph{(ii)},
\begin{align*}
    d(z) &= \limsup_{N\to\infty}\|T^N(z)-z\|\\
    &\le \limsup_{N\to\infty}r_a\big(\{T^n(x)\}, T^N(z)\big)\\
    &\le \limsup_{N\to\infty}\big(\limsup_{n\to\infty}\|T^n(x)-T^N(z)\|\big)\\
    &\le\gamma\limsup_{N\to\infty}\big(\lim_{j\to\infty}\sup_{n \ge j}\|T^{n-N}(x)-z\|\big)^{\alpha_N}\\
    &=\gamma\limsup_{N\to\infty}\big(\lim_{j\to\infty}\sup_{n \ge j-N}\|T^n(x)-z\|\big)^{\alpha_N}\\
    &= \gamma\limsup_{N\to\infty}\big(\limsup_{n \to \infty}\|T^n(x)-z\|\big)^{\alpha_N}\\
    &=\gamma \cdot r_a(\{T^n(x)\},z).
\end{align*}

Thus, 
\[
d(z)\le\gamma^2\tilde N(X) d(x)=\eta  d(x), 
\]
where $\eta=\gamma^2\tilde N(X)<1$. Define a sequence $\{x_n\}$ by choosing an arbitrary point $x_1\in K$ and setting $x_{n+1}=z(x_n)$. Then
\[
d(x_n)\leq \eta^{n-1}d(x_1)\longrightarrow 0
\qquad\text{as } n\to\infty.
\]
Moreover, the sequence $\{x_n\}$ is Cauchy. Indeed, for every $k\in\mathbb N$,
\[
\|x_{n+1}-x_n\| \leq \|x_{n+1}-T^k(x_n)\| + \|T^k(x_n)-x_n\|.
\]
Hence,
\[
\begin{split}
\|x_{n+1}-x_n\|&\leq\limsup_{k\to\infty}\|x_{n+1}-T^k(x_n)\| + \limsup_{k\to\infty}\|T^k(x_n)-x_n\|\\[2mm]
& =r_a({T^j(x_n)},x_{n+1}) + d(x_n)\\[2mm]
&\leq\gamma\tilde N(X)\cdot d(x_n) + d(x_n)\\[2mm]
&\leq\bigl(1+\gamma\tilde N(X)\bigr)\eta^{n-1}d(x_1).
\end{split}
\]
So, for any $q \geq 1$, we have
\[
\begin{split}
\|x_{n+q}-x_n\|&\leq \sum_{j=n}^{n+q-1}\|x_{j+1}-x_j\|\\
&\leq \bigl(1+\gamma\tilde N(X)\bigr)d(x_1)\sum_{j=n}^{n+q-1}\eta^{j-1}\\
&\leq\frac{\bigl(1+\gamma\tilde N(X)\bigr)d(x_1)}{(1-\eta)}\eta^{n-1}\longrightarrow0, \quad \text{as}\; n\to\infty.
\end{split}
\]
Therefore, there exists $y=\lim_{n\to\infty}x_n$. Furthermore,
\[
\begin{aligned}
\|y-Ty\|
&\leq \|y-x_n\|+\|x_n-T^N(x_n)\| \\
&\quad +\|T^N(x_n)-T(x_n)\|
+\|T(x_n)-T(y)\| \\
&\leq \|y-x_n\|+\|x_n-T^N(x_n)\| \\
&\quad +\gamma\|T^{N-1}(x_n)-x_n\|^{\alpha_1} +\gamma\|x_n-y\|^{\alpha_1}.
\end{aligned}
\]
Taking $\limsup_{N\to\infty}$, we obtain
\[
\|y-Ty\|\leq\|y-x_n\| + d(x_n) + \gamma d(x_n)^{\alpha_1} + \gamma\|x_n-y\|^{\alpha_1}.
\]
Finally, letting $n\to\infty$, we conclude that
\[
\|y-Ty\|=0.
\]
Hence,
\[
Ty=y.
\]
Since $K$ is closed and $\{x_n\}\subset K$, $y\in K$ and the proof is completed.
\end{proof}



\begin{proof}[\textbf{Proof of Theorem \ref{thm:main02}}]
Fix $x\in E$ and, for each $m\in\mathbb N$, consider the sequence
$(T^n x)_{n\geq m}$. Set
$C_m:=\overline{\operatorname{conv}}\{T^n x:n\geq m\}$.
Applying Lemma \ref{lem:casini} to $(T^n x)_{n\geq m}$, we obtain $y_m\in C_m$ such that $$r_a((T^n x)_{n\geq m},y_m) \leq\widetilde N(X)D((T^n x)_{n\geq m})$$ and $\|y_m-y\|\leq r_a((T^n x)_{n\geq m},y)$ for every $y\in X$.

Since removing finitely many terms does not change either $r_a$ or $D$, it follows that $r_a((T^n x),y_m)\leq\widetilde N(X)D((T^n x))$ and $\|y_m-y\|\leq r_a((T^n x),y)$ for every $y\in X$.

Since $y_m\in C_m\subset\overline{\operatorname{conv}}K$ for every $m$, and $K$ is bounded, the sequence $(y_m)$ is bounded. Moreover, since $X$ has uniform normal structure, $X$ is reflexive. Thus, by passing to a subsequence if necessary, we may assume that $y_{m_j}\rightharpoonup z(x)$ for some $z(x)\in X$.

Fix $m\in\mathbb N$. For all sufficiently large $j$, $m_j\geq m$, and hence $y_{m_j}\in C_{m_j}\subset C_m$. Since $C_m$ is closed and convex, it is weakly closed, and therefore $z(x)\in C_m$. Since $m$ is arbitrary, $z(x)\in\bigcap_{m\geq1}C_m$. On the other hand, we know from \cite{Bruck} that
\[
\bigcap_{k=1}^\infty \overline{\text{conv}}(\{T^j x : j \geq k\}) = \overline{\mathrm{conv}}(\omega_w(x)).
\]
Since $x\in E$, property $(P)$ yields $\omega_w(x)\subset E$.
As $E$ is closed and convex,
$\overline{\operatorname{conv}}\omega_w(x)\subset E$, and consequently
$z(x)\in E$.

The function $y\mapsto r_a((T^n x),y)$ is convex and norm continuous and,
therefore, weakly lower semicontinuous. Hence,
\[
r_a((T^n x),z(x))
\leq\liminf_{j\to\infty}r_a((T^n x),y_{m_j})
\leq\widetilde N(X)D((T^n x)).
\]
Furthermore, for each $y\in X$, the function $u\mapsto\|u-y\|$ is weakly
lower semicontinuous. Therefore,
\[
\|z(x)-y\|
\leq\liminf_{j\to\infty}\|y_{m_j}-y\|
\leq r_a((T^n x),y).
\]

Thus, for each $x\in E$, we have constructed a point $z(x)\in E$ satisfying the same estimates as those provided by Lemma \ref{lem:casini}. Repeating the argument used in the proof of Theorem \ref{thm:main1}, we obtain 
\[
d(z(x))\leq\gamma^2\widetilde N(X)d(x).
\]
Since $\gamma<\sqrt{N(X)}$, we have $\gamma^2\widetilde N(X)<1$. Choose $x_1\in E$ and define inductively $x_{n+1}:=z(x_n)$. Then $x_n\in E$ for every $n$. The same argument as in the proof of Theorem \ref{thm:main1} shows that $(x_n)$ converges to a fixed point $p$ of $T$. Since $E$ is closed, $p\in E$. Therefore, $T$ has a fixed point in $E$.
\end{proof}



\begin{proof}[\textbf{Proof of Theorem \ref{thm:main03}}]
    By Zorn's lemma, following the argument in \cite{XuExistence1991}, there exists a set $E$ that is minimal, with respect to inclusion, among all nonempty closed convex sets satisfying \eqref{prop:P}. Choose $\gamma$ such that $1<\gamma<\sqrt{N(X)}$. Since $k_n\to0$ and $\gamma>1$, there exists $N\in\mathbb{N}$ such that $1+k_n < \gamma<[N(X)]^{1/2}$ for all $n \ge N$. So,
    \[
        \|T^nx-T^ny\| \le (1+k_n)\|x-y\|^{\alpha_n} \le \gamma\|x-y\|^{\alpha_n}, \,\,\forall x,y\in K\textrm{ and }\forall n\ge N.
    \]
    Thus, being $\beta_n=\alpha_{Nn}$, we have
    \[
        \|(T^N)^nx-(T^N)^ny\| \le (1+k_{Nn})\|x-y\|^{\beta_n} \le \gamma\|x-y\|^{\beta_n},\]
        for all $x,y\in K$ and for all $n\in \mathbb{N}$.  
    In particular, $T^N$ is uniformly asymptotically Hölder $\gamma$-Lipschitzian mapping with $\gamma< [N(X)]^{1/2}$. Furthermore, it is obvious that $\omega_w^{T^N}(x)\subset\omega_w^{T}(x)\subset E$ for all $x\in E$ (Here, $\omega_w^T(x)$ denotes $\omega_w(x)$ with respect to the mapping $T$.). By Theorem \ref{thm:main02}, $T^N$ has a fixed point $z\in E$.\\

    \noindent Define
    \[
    D:=\{z,Tz,\dots,T^{N-1}z\}.
    \]

    \noindent Note that, for all $i=0,1,\dots,N-1$, $T^iz=T^{jN+i}z$, for all $j\in\mathbb{N}$. So, $T^iz\in\omega(z)$ and then $D\subset\omega(z)\subset\omega_w(z)\subset E$.

    \noindent Assume $\operatorname{diam}(D) > 0$. Now, $\overline{\operatorname{conv}} (D)$ is compact (since $D$ is finite) and convex. Thus, it has a normal structure. Therefore, there exists $u \in \overline{\operatorname{conv}}(D) \subset E$ that is not diametral, that is, such that
    \begin{equation}\label{prop:d}
    d:=\sup_{y\in D}\|u-y\|<\operatorname{diam}(D).
    \end{equation}
    Let 
    \[
    C:=\left\{z\in E:\,\sup_{y\in D}\|z-y\| \le d\right\}.
    \]
    Note that $C\neq\varnothing$ ($u\in C$), closed and convex. Furthermore, $C$ satisfies (\ref{prop:P}). Indeed, let $x\in C$ and $v\in\omega_w(x)$, that is, $T^{n_i}x\rightharpoonup v$ for some $n_i\to\infty$. Then, for all $y\in D$

    \[
    \|v-y\| \le \liminf_{i\to\infty}\|T^{n_i}(x)-y\| \le \limsup_{i\to\infty}\|T^{n_i}(x)-y\|.
    \]
    Note also that $T(D) = D$ and therefore, $T^n(D)=D$, for all $n\in\mathbb{N}$. Thus, for each natural $n$, there exists $y_n\in D$ such that $T^n(y_n)=y$ . It follows that

	\[
	   \begin{split}
	   \|v-y\| &\le \limsup_{i\to\infty}\|T^{n_i}(x)-T^{n_i}(y_{n_i})\| \le \limsup_{i\to\infty}(1+k_{n_i})\|x-y_{n_i}\|^{\alpha_{n_i}}\\
	   &\le \limsup_{i\to\infty}(1+k_{n_i})d^{\alpha_{n_i}}=d.
	   \end{split}	
	\]
    That is, $\displaystyle\sup_{y\in D}\|v - y\| \le d$, which gives us that $v\in C$. However, $C$ is a proper subset of $E$, because of (\ref{prop:d}). This contradicts the minimality of $E$. Therefore, $\operatorname{diam}(D)=0$ and $Tz=z$.

\end{proof}

\begin{remark}
Theorem \ref{thm:main03} can also be deduced from
\cite[Corollary~2]{LiSims2002}. Indeed, let
$L=\operatorname{diam}(K)$ and set
\[
\varepsilon_n
:=
\sup_{0\le t\le L}
\left|(1+k_n)t^{\alpha_n}-t\right|.
\]
Since $\alpha_n\to1$ and $k_n\to0$, the uniform
convergence of $t^{\alpha_n}$ to $t$ on $[0,L]$
implies that $\varepsilon_n\to0$. Hence, for every
$x\in K$,
\[
\limsup_{n\to\infty}\sup_{y\in K}
\left(
\|T^n x-T^n y\|-\|x-y\|
\right)
\le
\lim_{n\to\infty}\varepsilon_n
=0.
\]
Thus, $T$ is of asymptotically nonexpansive type.
Moreover, the estimate
\[
\|Tx-Ty\|
\le (1+k_1)\|x-y\|^{\alpha_1},
\qquad x,y\in K,
\]
ensures that $T$ is continuous. Therefore,
\cite[Corollary~2]{LiSims2002} yields a fixed point
of $T$ in $K$. The proof given above provides
an alternative derivation through Theorem \ref{thm:main02}.
\end{remark}

\section{Examples}\label{sec:exemplos}

\begin{example}
This example illustrates how Lemma~2.2 can strictly reduce
the asymptotic diameter while preserving the asymptotic
radius at every point of the closed linear span of the
original sequence.

Let $(e_j)$ be the canonical basis of $\ell_2$, and define
\[
x_{2j-1}=e_j,
\qquad
x_{2j}=-e_j,
\qquad j\ge1.
\]
For each fixed $m$, the vectors $x_n$ and $x_m$ are
orthogonal for all sufficiently large $n$. Hence
\[
\limsup_{n\to\infty}\|x_n-x_m\|=\sqrt{2},
\]
and therefore
\[
D((x_n))
=
\limsup_{m\to\infty}\limsup_{n\to\infty}
\|x_n-x_m\|
=
\sqrt{2}.
\]
On the other hand, every tail of $(x_n)$ contains a pair
$e_j,-e_j$, whose distance is $2$. Since $\|x_n\|=1$
for every $n$, it follows that
\[
d_a((x_n))=2.
\]

Now consider the subsequence $u_j=x_{2j-1}=e_j$.
Since $\|e_i-e_j\|=\sqrt{2}$ whenever $i\ne j$, we have
\[
d_a((u_j))=\sqrt{2}=D((x_n))<d_a((x_n)).
\]
Moreover,
\[
Y:=\overline{\operatorname{span}}\{x_n:n\ge1\}
=\ell_2.
\]
For every $y=(y_j)\in\ell_2$, we have $y_j\to0$, and thus
\[
\|\pm e_j-y\|^2
=
1+\|y\|^2\mp 2y_j
\longrightarrow
1+\|y\|^2.
\]
Consequently,
\[
r_a((u_j),y)
=
r_a((x_n),y)
=
\sqrt{1+\|y\|^2},
\qquad y\in Y.
\]

Thus, both conclusions of Lemma \ref{lem:extraction} hold for $(u_j)$,
with a strict reduction in the asymptotic diameter.
The preservation of the asymptotic radii allows radius
estimates obtained for the subsequence to be transferred
to the original sequence at every point of $Y$.
\end{example}

Now, we show some examples in $X=\ell_2$. It is known that $N(\ell_2)=\sqrt{2}$ (See \cite{Benavides}). In particular, $\ell_2$ has uniformly normal structure. 
\begin{example}\label{ex:non-lips}
We give an example of a non-Lipschitzian uniformly
asymptotically H\"older $k$-Lipschitzian mapping with $1<k<\sqrt{N(\ell_2)}$. Let $\alpha_n=\frac{n+1}{n+2}$, $p_1=\alpha_1=\frac23$, $p_i=\frac{\alpha_i}{\alpha_{i-1}}$ with $i\geq2$, and set
\[
    K_0=
    \left\{
        (t_i)\in\ell_2:
        t_i\geq0,\quad
        \sum_{i=1}^{\infty}t_i\leq\frac1{\sqrt2}
    \right\}.
\]
Define $T_0:K_0\to K_0$ by
\[
    T_0(t_1,t_2,\ldots)
    =
    \frac14
    \sum_{i=1}^{\infty}t_i^{p_i}e_{i+1}.
\]
Note that $T_0$ is well-definid. Since $\operatorname{diam}(K_0)\leq1$, a direct computation using $\prod_{j=1}^{n}p_j=\alpha_n$, $1\leq
    \prod_{j=i}^{i+n-1}p_j
    =
    \frac{\alpha_{i+n-1}}{\alpha_{i-1}}
    \leq\frac32
    \quad(i\geq2)$, gives
\[
    \|T_0^n(x)-T_0^n(y)\|_2
    \leq
    \frac{\sqrt{13}}8
    \|x-y\|_2^{\alpha_n},
    \qquad x,y\in K_0,\quad n\geq1.
\]

\noindent For $\lambda>1$, set $K_\lambda=\lambda K_0$ and define
\[
    T_\lambda(x)
    =
    \lambda T_0\left(\frac{x}{\lambda}\right).
\]
Then $T_\lambda^n(x)
    =
    \lambda T_0^n\left(\frac{x}{\lambda}\right)$, and hence
\[
\begin{aligned}
    \|T_\lambda^n(x)-T_\lambda^n(y)\|_2
    &\leq
    \frac{\sqrt{13}}8
    \lambda^{1-\alpha_n}
    \|x-y\|_2^{\alpha_n}\\
    &\leq
    \frac{\sqrt{13}}8
    \lambda^{1/3}
    \|x-y\|_2^{\alpha_n}.
\end{aligned}
\]
Thus we may take $k_\lambda
    =
    \frac{\sqrt{13}}8\lambda^{1/3}$. Since $N(\ell_2)=\sqrt2$, choosing $\frac{512}{13\sqrt{13}}
    <
    \lambda
    <
    \frac{512\cdot2^{3/4}}{13\sqrt{13}}$ yields $1<k_\lambda<2^{1/4}
    =
    \sqrt{N(\ell_2)}$.

\noindent Moreover, $T_\lambda$ is not Lipschitzian. Indeed, for $0<h<\frac{\lambda}{\sqrt{2}}$, $T_\lambda(he_1)
    =
    \frac14\lambda^{1/3}h^{2/3}e_2$, so
\[
    \frac{
        \|T_\lambda(he_1)-T_\lambda(0)\|_2
    }{
        \|he_1\|_2
    }
    =
    \frac14\lambda^{1/3}h^{-1/3}
    \longrightarrow\infty
\]
as $h\to0^+$. Finally,
\[
    \operatorname{Fix}(T_\lambda)=\{0\}.
\]
Hence $T_\lambda$ is a non-Lipschitzian uniformly asymptotically
H\"older $k_\lambda$-Lipschitzian mapping with $1<k_\lambda<\sqrt{N(\ell_2)}$ and a unique fixed point.
\end{example}

\begin{remark}
  Note that Theorem~3.1 in \cite{CM} is not applicable to this example, since neither $T_0$ nor $T_\lambda$ is Lipschitzian and, consequently, neither mapping is uniformly $k$-Lipschitzian.
\end{remark}

\begin{example}
The assumption $\alpha_n\to 1$ cannot, in general, be replaced by a fixed H\"older exponent $\alpha<1$. Indeed, let $X=\ell_2$ and choose $\frac14<r<\frac{1}{2\sqrt{2}}$. Set $K=rB_{\ell_2}$ and define $T:K\to K$ by (see \cite{kakutani1943})
\[
    T(x_1,x_2,\ldots)
    =
    \left(
        \sqrt{r^2-\|x\|_2^2},
        x_1,x_2,\ldots
    \right).
\]
For every $x\in K$, $\|T(x)\|_2^2
    =
    r^2-\|x\|_2^2+\|x\|_2^2
    =
    r^2$, so $T(K)\subset K$.

\noindent Let $S:\ell_2\to\ell_2$ be the right-shift operator. Since
$\|T(x)\|_2=r$, we have
\[
    T^n(x)=S^{n-1}(T(x)),
    \qquad n\ge1.
\]
Hence, since $S$ is an isometry, $\|T^n(x)-T^n(y)\|_2
    =
    \|T(x)-T(y)\|_2$ for all $x,y\in K$ and $n\ge1$.

\noindent Now, using $|\sqrt{a}-\sqrt{b}|^2\le |a-b|$, $a,b\ge0$, we obtain
\[
\begin{aligned}
    \|T(x)-T(y)\|_2^2
    &\le
    \left|\|x\|_2^2-\|y\|_2^2\right|
    +
    \|x-y\|_2^2
    \\
    &\le
    2r\|x-y\|_2+\|x-y\|_2^2.
\end{aligned}
\]
Since $\|x-y\|_2\le2r$, it follows that $\|T(x)-T(y)\|_2^2
    \le
    4r\|x-y\|_2$. Therefore,
\[
    \|T^n(x)-T^n(y)\|_2
    \le
    2\sqrt{r}\,
    \|x-y\|_2^{1/2},
    \qquad n\ge1.
\]
Thus $T$ is uniformly H\"older-Lipschitzian with the fixed exponent $\alpha=\frac12$ and constant $k=2\sqrt{r}$. Since $N(\ell_2)=\sqrt{2}$ (See \cite{Benavides}), our choice of $r$ yields $1<k=2\sqrt{r}
    <
    2^{1/4}
    =
    \sqrt{N(\ell_2)}$.
    
\noindent Finally, $T$ is fixed-point free. Indeed, if $T(x)=x$, then
$\|x\|_2=r$, and hence $x=T(x)=Sx$. This forces $x=0$, contradicting $\|x\|_2=r>0$. Therefore, $\operatorname{Fix}(T)=\varnothing$.

\noindent Hence, even under the condition $1<k<\sqrt{N(X)}$, the fixed point conclusion may fail when the sequence
$\alpha_n\to1$ is replaced by a fixed exponent $\alpha<1$.
\end{example}
Before presenting our next example, we recall that, for Banach spaces $X$
and $Y$, $X\oplus_2Y$ denotes their $\ell_2$-direct sum, that is, the space
$X\times Y$ endowed with the norm
\[
\|(x,y)\|
=
\left(\|x\|_X^2+\|y\|_Y^2\right)^{1/2}.
\]

\begin{example}
\label{ex:fixed-point-free-non-lipschitz}

We now give a fixed-point free example in the genuinely non-Lipschitzian
setting. Let $\alpha_n=\frac{n+1}{n+2}$, $n\in\mathbb{N}$, so that $\alpha_n\uparrow1$. First, let $C_0=B_{\ell_2}^{+}=\{x\in\ell_2:\, x_i\ge 0,\,\|x\|_2\le1\}$ and denote by $S$ the right-shift operator on $\ell_2$. Let $G:C_0\to C_0$ be the classical Goebel--Kirk--Thele mapping, obtained by normalizing the Kakutani mapping, i.e., 
\[
    G(x)=\frac{\Phi(x)}{\|\Phi(x)\|_2},
\]
where $\Phi(x)=(1-\|x\|_2)e_1+Sx$ (See \cite{kakutani1943} and \cite{thele1974}). Since
\[
    \|\Phi(x)\|_2^2
    =
    (1-\|x\|_2)^2+\|x\|_2^2
    \geq\frac12,
\]
the map $G:C_0\to C_0$ is well defined. Moreover,
\[
    G^n(x)=S^{n-1}G(x),
\]
and $\|G^n(x)-G^n(y)\|_2
    \leq2\|x-y\|_2$, $n\geq1$. It is also immediate that $\operatorname{Fix}(G)=\varnothing$.

\noindent Set $C=\frac14C_0$, $A(x)=\frac14G(4x)$. Then
$\operatorname{Fix}(A)=\varnothing$, $\|A^n(x)-A^n(y)\|_2
    \leq2\|x-y\|_2$, and $\operatorname{diam}(C)=\frac{\sqrt2}{4}$.

For the second component, let $B_0:K_0\to K_0$ be the mapping introduced in Example \ref{ex:non-lips}, namely
\[
    K_0=
    \left\{
        (t_i)\in\ell_2:
        t_i\geq0,\quad
        \sum_{i=1}^{\infty}t_i\leq\frac1{\sqrt2}
    \right\},
\]
\[
    B_0(t_1,t_2,\ldots)
    =
    \frac14\sum_{i=1}^{\infty}t_i^{p_i}e_{i+1},
\]
where $p_1=\frac23$, $p_i=\frac{\alpha_i}{\alpha_{i-1}}$, $i\geq2$. As shown above,
\[
    \|B_0^n(u)-B_0^n(v)\|_2
    \leq
    \frac{\sqrt{13}}8
    \|u-v\|_2^{\alpha_n},
\]
and $B_0$ is not Lipschitzian. Set $D=\frac12K_0$, $B(u)=\frac12B_0(2u)$. Then
\[
    \|B^n(u)-B^n(v)\|_2
    \leq
    \frac{\sqrt{13}}8
    \|u-v\|_2^{\alpha_n},
\]
$B$ is not Lipschitzian, and $\operatorname{diam}(D)\leq\frac12$.

Now consider $X=\ell_2\oplus_2\ell_2$ and $K=C\times D$. Since $X$ is isometrically isomorphic to $\ell_2$, $N(X)=\sqrt2$. Furthermore, $\operatorname{diam}(K)^2\leq
    \frac18+\frac14
    =
    \frac38$, so $\operatorname{diam}(K)<1$.

Define $T:K\to K$,
\[
    T(x,u)=(A(x),B(u)).
\]
For $z=(x,u)$, $w=(y,v)$ and
$d=\|z-w\|\leq1$, we have
\[
\begin{aligned}
    \|T^n z-T^n w\|^2
    &=
    \|A^n x-A^n y\|_2^2
    +
    \|B^n u-B^n v\|_2^2\\
    &\leq
    \left(
        4+\frac{13}{64}
    \right)d^{2\alpha_n}.
\end{aligned}
\]
Hence
\[
    \|T^n z-T^n w\|
    \leq
    \frac{\sqrt{269}}8
    \|z-w\|^{\alpha_n}
\]
for every $z,w\in K$ and $n\geq1$. Moreover, $\operatorname{Fix}(T)=\varnothing$, since $\operatorname{Fix}(A)=\varnothing$. Finally, $T$ is not
Lipschitzian, because for $h$ small enough and any fixed $x_0\in C$,
\[
    \frac{
        \left\|
            T\left(x_0,\frac h2e_1\right)-T(x_0,0)
        \right\|
    }{
        \left\|
            \left(x_0,\frac h2e_1\right)-(x_0,0)
        \right\|
    }
    =
    \frac14h^{-1/3}
    \longrightarrow\infty.
\]

\noindent Thus $T$ is a fixed-point free, non-Lipschitzian uniformly
asymptotically H\"older $k$-Lipschitzian mapping with $k=\frac{\sqrt{269}}8
    >
    \sqrt{N(X)}$.
    
This complements the preceding examples by showing that, beyond the
threshold $k<\sqrt{N(X)}$, the fixed point conclusion may fail even for
genuinely non-Lipschitzian mappings with $\alpha_n\to1$.
\end{example}


\section*{Acknowledgements}
The authors would like to express their sincere gratitude to Professor C.~S.~Barroso for introducing them to this topic and for conceiving the overall research project.

\section*{Data Availability}
Data sharing is not applicable to this article as no datasets were generated or analyzed during the current study.

\section*{Declaration of Competing Interest}
The authors declare that they have no known competing financial interests or personal relationships that could have appeared to influence the work reported in this paper.

\section*{Use of Generative AI}

During the preparation of this work, the authors used generative AI tools to assist with language refinement, bibliographic search, as well as preliminary consistency checks and critical evaluation of the arguments. All outputs were independently cross-checked. The authors reviewed and edited all material and assume full responsibility for the final content and mathematical accuracy of the manuscript.


\end{document}